\documentclass[reqno]{amsart}  
\usepackage{color}              
\usepackage[usenames,dvipsnames]{xcolor}
\usepackage{graphicx}
\usepackage[]{amsmath}
\usepackage{amsthm}
\usepackage{amssymb}
\usepackage{bbm}
\usepackage{comment}
\usepackage{microtype}
\usepackage[textsize=small]{todonotes}
\usepackage[colorlinks,citecolor=blue,urlcolor=blue,linkcolor=blue,pagecolor=blue,linktocpage=true,backref=true]{hyperref}

\usepackage{enumerate}
\newenvironment{enumeratei}{\begin{enumerate}[\upshape (i)]}{\end{enumerate}}
\newenvironment{enumeratea}{\begin{enumerate}[\upshape (a)]}{\end{enumerate}}

\definecolor{MyDarkBlue}{rgb}{0,0.08,0.50}
\definecolor{BrickRed}{rgb}{0.65,0.08,0}

\hypersetup{
colorlinks=true,       
    linkcolor=MyDarkBlue,          
    citecolor=BrickRed,        
    filecolor=red,      
    urlcolor=cyan           
}

\newtheorem{Lemma}{Lemma}[section]
\newtheorem{Proposition}[Lemma]{Proposition}

\newtheorem{Theorem}[Lemma]{Theorem}

\newtheorem{Condition}[Lemma]{Condition}

\theoremstyle{definition}
\newtheorem{Remark}[Lemma]{Remark}

\newcommand{\prob}{\mathbb{P}}
\newcommand{\Pv}{\mathbb{P}}

\newcommand{\E}{\mathbb{E}}
\newcommand{\Ev}{\mathbb{E}}

\newcommand{\CC}{\mathcal{C}}

\newcommand{\set}[1]{\left\{#1\right\}}

\newcommand{\expec}{\mathbb{E}}

\newcommand{\bfd}{{\boldsymbol d}}
\newcommand {\convd}{\stackrel{d}{\longrightarrow}}

\newcommand{\eqn}[1]{\begin{equation} #1 \end{equation}}
\newcommand{\eqan}[1]{\begin{align} #1 \end{align}}
\newcommand{\sss}{\scriptscriptstyle}
\newcommand{\Var}{{\rm Var}}

\newcommand{\op}{o_{\sss \prob}}

\newcommand {\convp}{\stackrel{\sss {\mathbb P}}{\longrightarrow}}
\newcommand {\convas}{\stackrel{\sss a.s.}{\longrightarrow}}
\newcommand*{\CMD}{{\mathrm{CM}}_n(\boldsymbol{d})}

\newcommand{\e}{{\mathrm e}}
\newcommand{\FR}{F_{\sss R}}
\newcommand{\fR}{f_{\sss R}}

\newcommand{\Wn}{L_n}
\newcommand{\nn}{\nonumber}

\numberwithin{equation}{section}
\newcommand{\Pn}{P_n}
\newcommand{\alphan}{\lambda_n}
\newcommand{\tn}{\bar{t}_n}
\newcommand{\Tcol}{T^{\sss{\rm (col)}}}
\newcommand{\barTcol}{\bar{T}^{\sss{\rm (col)}}}
\newcommand{\Scal}{{\mathcal S}}
\newcommand{\Mcal}{{\mathcal M}}

\newcommand{\R}{\mathbb{R}}
\newcommand{\N}{\mathbb{N}}

\newcommand{\indicator}[1]{\mathbbm{1}_{\set{#1}}}

\newcommand{\dd}{\mathrm{d}}
\newcommand{\cP}{\mathcal{P}}

\renewcommand{\emptyset}{\varnothing}

\newcommand{\refeq}[1]{\eqref{eq:#1}} 

\newcommand{\RvdH}[1]{\todo[color=Magenta,inline]{Remco: #1}}
\newcommand{\KJ}[1]{\todo[color=SkyBlue,inline]{Juli: #1}}

\def\ind{{\rm 1\hspace{-0.90ex}1}}

\newcommand{\cA}{\mathcal{A}}
\newcommand{\cF}{\mathcal{F}}
\newcommand{\cG}{\mathcal{G}}
\newcommand{\cL}{\mathcal{L}}
\newcommand{\cN}{\mathcal{N}}

\newcommand{\cV}{\mathcal{V}}
  
\newcommand{\bN}{\mathbb{N}}
\newcommand{\bR}{\mathbb{R}}

\newcommand{\BP}{{\sf BP}}

\newcommand{\CA}{\mathcal {A}}
\newcommand{\CB}{\mathcal {B}}

\newcommand{\CF}{\mathcal {F}}

\newcommand{\CI}{\mathcal {I}}

\newcommand{\CL}{\mathcal {L}}

\newcommand{\CP}{\mathcal {P}}

\newcommand{\CU}{\mathcal {U}}

\newcommand{\CH}{\mathcal {H}}

\newcommand*{\la}{\lambda}
\newcommand*{\La}{\Lambda}

\newcommand*{\Rb}{\mathbb R}

\newcommand*{\be}{\begin{equation}}
\newcommand*{\ee}{\end{equation}}
\newcommand*{\ba}{\begin{aligned}}
\newcommand*{\ea}{\end{aligned}}
\newcommand*{\barr}{\begin{array}{c}}
\newcommand*{\earr}{\end{array}}

\def \toindis  {\buildrel {d}\over{\longrightarrow}}
\def \toas     {\buildrel {a.s.}\over{\longrightarrow}}

\begin{document}

		\title[Epidemics and FPP]{The front of the epidemic spread\\and first passage percolation}

		\date{\today}
		\subjclass[2000]{Primary: 60C05, 05C80, 90B15.}
		\keywords{Flows, random graphs, random networks, epidemics on random graphs, first passage percolation,
		hopcount, interacting particle systems}

		\author[Bhamidi]{Shankar Bhamidi$^1$}
		\address{$^1$Department of Statistics, University of North Carolina, Chapel Hill.}
		\author[van der Hofstad]{Remco van der Hofstad$^3$}
		\address{$^3$Department of Mathematics and
		    Computer Science, Eindhoven University of Technology, P.O.\ Box 513,
		    5600 MB Eindhoven, The Netherlands.}

		\author[Komj\'athy]{J\'ulia Komj\'athy$^3$}
		\email{bhamidi@email.unc.edu, rhofstad@win.tue.nl,j.komjathy@tue.nl}

\begin{abstract}
In this paper we establish a connection between epidemic models on random networks with general infection times considered in \cite{BarRei13} and first passage percolation. Using techniques developed in \cite{BhaHofHoo12}, when each vertex has infinite contagious periods,  we extend results on the epidemic curve in \cite{BarRei13} from bounded degree graphs to general sparse random graphs with degrees having finite third moments as $n\to\infty$. We also study the epidemic trail between the source and typical vertices in the graph. This connection to first passage percolation can be also be used to study epidemic models with general contagious periods as in \cite{BarRei13} without bounded degree assumptions.   
\end{abstract}

	\maketitle

\section{Introduction and model}\label{s:int}
We consider the spread of an epidemic on the configuration model with i.i.d.\ infection times having a general continuous distribution, and an infinite contagious period for each vertex. We describe the link between first passage percolation (FPP) on sparse random graph models \cite{BhaHofHoo09b, BhaHofHoo12}, and general epidemics on the configuration model by Barbour and Reinert \cite{BarRei13}. The work in \cite{BhaHofHoo09b, BhaHofHoo12} is more general in terms of the graph models allowed, but more restrictive in terms of the epidemic process, requiring the assumption of infinite contagious periods and i.i.d.\ infection times, while the work in \cite{BarRei13} allows for more general epidemic processes, but assumes the graphs have bounded degrees. The main result, Theorem \ref{thm-epidemic-curve} below, extends \cite{BhaHofHoo09b, BhaHofHoo10, BhaHofHoo12} to the study of the epidemic curve in the spirit of \cite{BarRei13} by describing how the infection sweeps through the system. We also investigate the {\em epidemic trail}, namely the number of individuals that spread the infection from the source to the destination. Branching process approximations for the epidemic process and stable-age distribution theory for the corresponding branching processes developed by Jagers and Nerman \cite{Jage75,jagers1984growth,Ner81} play a critical role in the proof of the main result.

\subsection{Configuration model}
We first describe the model for the underlying network on which the epidemic process takes place. The configuration model $\CMD$ (see \cite{Boll01} or \cite[Chapters 7 and 10]{Hofs10c}) on $n$ vertices with degree sequence $\boldsymbol{d}_n=(d_1,\dotsc,d_n)$ is constructed as follows. Let $[n]:=\set{1,2,\ldots, n}$ denote the vertex set.  To each vertex $i\in [n]$, attach $d_i$ half-edges to that vertex with total degree $\CL_n=\sum_{i\in[n]} d_i$ assumed even (when the degrees $d_i$ are drawn independently from some common degree distribution $D$, $\CL_n$ may be odd; if so, select one of the $d_i$ uniformly at random and increase it by 1).

We number the half-edges in any arbitrary order from $1$ to $\CL_n$.
We start pairing them uniformly at random, i.e.,  we pick an arbitrary unpaired
half-edge and pair it to another unpaired half-edge chosen uniformly at random to form an
edge. Once paired, we remove both from the set
of unpaired half-edges and continue until all half-edges are paired.
We denote the resulting random multi-graph by $\CMD$. Although self-loops and multiple edges may occur, under weak assumptions on the degree sequence (satisfied via Condition \ref{cond:CMFinVar} below), their number is a tight sequence as $n\to\infty$ (see \cite{Jans09} or \cite{Boll01} for more precise results in this direction).

We consider the configuration model for general degree sequences $\bfd_n$, which may be either deterministic or random, subject to mild regularity conditions as $n\to\infty$.  To formulate these conditions, we think of $\bfd_n = (d_v)_{v\in [n]}$ as fixed and choose a vertex $V_n$ uniformly from $[n]$.  Then, the distribution of $d_{V_n}$ is the degree of a uniformly chosen vertex $V_n$, conditional on the degree sequence $\bfd_n$. To ensure that the majority of vertices are connected in the resulting graph, we assume throughout that $d_v\geq 2$ for each $v\in [n]$ (see e.g.\ \cite{JanLuc07} or or \cite[Chapter 10]{Hofs10c}). We make the following key assumption on the degree sequence:

\begin{Condition}[Degree regularity]\label{cond:CMFinVar}
The degrees $d_{V_n}$ satisfy $d_{V_n}\geq 2$ a.s.\ and, for some random variable $D$ with $\prob(D>2)>0$ and $\E(D^2\log^+(D))<\infty$,
	\begin{equation}
	d_{V_n}\convd D,
	\qquad
	\E(d_{V_n}^2) \to \E(D^2).
	\end{equation}
Furthermore,
	\begin{equation}
	\limsup_{n\to\infty} \E\big(d_{V_n}^2 \log^+(d_{V_n})\big)= \E(D^2\log^+(D)).
	\end{equation}
\end{Condition}

When $\bfd_n$ is itself random, we require that the convergences in Condition \ref{cond:CMFinVar} hold in probability. Next, we define the \emph{size-biasing} $D^\star_n$ of $d_{V_n}:=D_n$ by
	\begin{equation}\label{eq:BiasedD}
	\prob(D^\star_n=k)=\frac{(k+1)\prob(D_n=k+1)}{\E(D_n)}.
	\end{equation}
It is easily checked that uniform integrability following from Condition \ref{cond:CMFinVar} implies that $\Ev[D_n^\star] \to \Ev[D^\star]=\E[D(D-1)]/\E[D]<\infty$ where $D^\star$ is the corresponding size-biasing for $D$. The assumption $d_{V_n}\geq 2$ and non-vanishing variance $\Var(D)>0$ of the degrees implies that $\Ev[D^\star]>1$. 

\subsection{Epidemic model}
Let us now describe the infection model on $\CMD$. Since multiple edges and self-loops play no role in the dynamics, we replace multiple edges by a single edge and replace self-loops.
 We also view each edge $e = \set{u,v}$ in $\CMD$ as two directed edges $(u,v)$ and $(v,u)$.   We consider an SIR (Susceptible-Infected-Removed) process on $\CMD$. Fix a continuous distribution $G$ on $\Rb_+$.  At time $t=0$, start the infection at a uniformly chosen vertex $V_n$. Each infected vertex infects its neighbors at times that are i.i.d.\ with distribution $G$ after the vertex is infected.  This can be modelled by adding i.i.d.\ edge lengths $X_e\sim G$ for every directed edge $e=(v,u)$ between a neighbors $u$ of the vertex $v$ $\in \CMD$. If each vertex $v$ has an i.i.d.\ contagious period $C_v\le  \infty$ after which it recovers,  then once $v$ gets infected only those neighbours $u$ of $v$ get infected that have an infection time $X_{(v,u)} < C_v$. We denote the (possibly non-proper) tail distribution function of $C$ by $\bar H$, i.e.\ $\bar H(x)=\Pv(C>x)$. Finally we assume that if a vertex has been infected once it cannot be infected again and thus transmits infection to its neighbours at most once. We let $(\cF_n(t))_{t\geq 0}$ denote this epidemic process. Here for any fixed $t\geq 0$, $\cF_n(t)$ contains the entire sigma-field of the process till time $t$, thus containing information not only of the set and number of infected individuals by time $t$, but also of the entire sequence of transmissions until this time. We use $|\cF_n(t)|$ for the total number of infected individuals by time $t$, and $|\CA_n(t)|$ for the total size of the \emph{coming generation}: those vertices who are not yet infected but have an infectious neighbour at time $t$ in the graph who is going to infect them some time after $t$. Later we will define a related process $(\widetilde \CF_n(t)), \widetilde \CA_n(t))_{t\geq 0}$ representing the collection of individuals that \emph{would infect} a fixed target individual $w$ \emph{by time $t$ if} were the epidemic to start from them, and the corresponding \emph{coming generation} in this process.
We call this the \emph{backward infection process}, see Section \ref{ss:labelling} for a precise definition.
\section{Results}\label{s:results} 
In this section, we state our main results.
Let $\Pn(s)$ denote the proportion of vertices infected by time $s$, i.e., 
	\eqn{
	\Pn(s)= \frac1n\sum_{w\in[n]} \ind \set{\text{vertex } w \text{ infected by time } s}.
	}
We also investigate the \emph{number of infected individuals} on the path from the initial source of the infection to other vertices in $\CMD$. Since the infection times are continuous random variables, there is a.s.\ a \emph{unique} path that realizes the infection between $V_n$ and any other fixed vertex $w\in [n]$, which we call the \emph{infection trail} to vertex $w$. 
We let $H_n(w)$ denote the number of infectives along the trail to $w$ (including $V_n$ and $w$), and  define
	\eqn{
	\Pn(s, h)= 
	\frac1n\sum_{w\in [n]} \ind \set{\text{vertex } w\text{ infected by time } s,\text{ and } H_n(w)\leq h}.
	}
Now fix $n\geq 1$. In Section \ref{s:proofs} we describe how to couple the \emph{epidemic process} and the \emph{backward infection process} $(\cF_n(t), \widetilde \cF_n(t))_{t\geq 0}$  to two independent Crump-Mode-Jagers processes $(\BP_n(t), \widetilde \BP_n(t))_{t\geq 0}$ where each individual from the first generation onwards produces a random number of children with distribution $D_n^\star$ with birth times that are i.i.d.\ variables with cumulative distribution function $G$, and with a possibly finite contagious period $C_v$ whose tail distribution we write as $\bar H$. The root has a slightly different offspring distribution from the rest of the population. Recall that $|\CA_n(t)|, |\widetilde \CA_n(t)|$ stands for the \emph{coming generation} in the infection processes. Condition \ref{cond:CMFinVar} and standard results \cite{Jage75,JagNer84} which we describe in Section \ref{s:proofs} imply that there exists a constant $\la_n>0$ and limit random variables $W_n, \widetilde W_n > 0$ 
 a.s.\ such that $\exp\{-\lambda_n t\} (|\CA_n(t)|, |\widetilde \CA_n(t)|)  \convas (W_n, \widetilde W_n)$ as $t\to\infty$ (see \eqref{eq:W}),
where $\lambda_n$ satisfies the equation
	\begin{equation}
	\label{eqn:mainparameters-lamb}
		 \Ev(D^\star_n) \int_{\bR_+} \e^{-\la_n x}\bar H(x)dG(x) = 1
	\end{equation}
and further 
	\[
	(W_n, \widetilde W_n) \convd (W, \widetilde W), \qquad 
	\lambda_n\to \lambda, \qquad\mbox{ as } n\to\infty,
	\]
where $W, \widetilde W$ are the corresponding limit random variables for the branching processes $(\BP(t), \widetilde \BP(t))_{t\geq 0}$ described below in Section \ref{sec:exploration-shortest-path} and \ref{s:BP} and $\lambda$ satisfies \eqref{eqn:mainparameters-lamb} with $D_n^\star$ replaced by $D^\star$.
Let $\Lambda$ be a standard Gumbel random variable independent of  $(S,\widetilde S)\stackrel{d}{:=} (-\frac{1}{\lambda} \log{W}, -\frac{1}{\lambda} \log \widetilde W)$. Define the function
	\eqn{
	\label{P(t)-def}
	P(t)=\prob\big(\widetilde S -\Lambda/\la+c \leq t\big), \qquad t\in \bR.
	}
Finally let $\Phi(\cdot)$ denote the standard normal cdf. 

Our main theorem describes the asymptotics for the functions $P_n(t), P_n(t,h)$ and shows that these functions follow a deterministic curve with a random time-shift corresponding to the initial phase of the infection:

\begin{Theorem}[Epidemic curve]
\label{thm-epidemic-curve}
Consider the epidemic spread with i.i.d.\ continuous infection times on the configuration model $\CMD$ and infinite contagious periods. 
Assuming condition \eqref{cond:CMFinVar}, for each fixed $t \in \bR$, the proportion of infected individuals satisfies
	\eqn{
	\label{Pn(t)-conv}
	\Pn\left(t+\frac{\log{n}}{\lambda_n}\right)\toindis P(t-S),
	}
Further, 
	\eqn{
	\label{Pn(t,h)-conv}
	\Pn\Big(t+\frac{1}{\lambda_n}\log{n}, \alpha_n\log{n}+x\sqrt{\beta\log{n}}\Big)
	\toindis P\big(t-S\big)\Phi(x),
	}
where $\alpha_n$ and $\beta$ are constants arising from the branching process $\BP_n(\cdot)$ and $\BP(\cdot)$, and are defined below \eqref{Ln-lim-pf}.
\end{Theorem}



\begin{Remark}
Theorem \ref{thm-epidemic-curve} implies that the epidemic sweeps through the graph in an almost deterministic fashion, where the dependence on the initial start of the epidemic only appears in the random shift $S$ in \eqref{Pn(t)-conv}. Further, \eqref{Pn(t,h)-conv} implies that the number of infectives needed to reach a typical vertex in the graph is aymptotically {\em independent} of the time at which the vertex is infected. Much information can be read off from the shape of the curve $t\mapsto P(t)$. For example, the fact that in the initial phase, the infection grows {\em exponentially} is related to the fact that 
$P(t)$ decays exponentially at $t=-\infty$, which, in turn, follows from the fact that $\prob(-\Lambda/\la+c \leq t)$ decays exponentially for $t$ large and negative.
\end{Remark}

\begin{Remark}
We believe this connection between first passage percolation and epidemic models used to prove the above result can easily be generalized to the case with finite contagious times. In this regime,  the forward and the backward branching process have {\em identical} Malthusian rates of growth but different limit random variables, see Section \ref{ss:forwardbackward}. This would extend results in \cite{BarRei13} where one assumes that the degree of all vertices is bounded by some constant $K$ to the general configuration model satisfying Condition \ref{cond:CMFinVar}.
\end{Remark}

\section{Discussion}
Here we briefly describe the connection between our work and related work. 
\begin{enumeratea}
	\item {\bf Epidemic models on networks:} There is an enormous literature on general epidemic models, their behavior on various network models and their connections to other dynamic process; see \cite{BarBartVes08,NewBarWat06} and the references therein for a description of the motivations from statistical physics and see  \cite{DraMas10,Durr06,Ald13} and references therein for pointers to more rigorous results. First passage percolation or shortest path problems play an integral role in our study and we use results in \cite{BhaHofHoo12} for the analysis of such processes on general sparse graph models with general edge distributions. 

	\item {\bf Connection to the results of Barbour and Reinert:}
	In \cite{BarRei13}, the authors determine the epidemic curve for a mean-field model with a Poisson number of infections. This case is equivalent to the infection spread on the \emph{Erd\H{o}s-R\'enyi random graph}. They generalize this to multi-type epidemics, and conclude that a similar result holds true for the configuration model where every vertex has degree bounded $K$ for some fixed constant $K\geq 1$. This restriction allows them to consider infection rates with arbitrary dependence on the number of possible infections created by a vertex. They use an associated multi-type branching processes for their analysis. 
Using the connection to first passage percolation, we show that similar results can be derived for any degree distribution satisfying Condition \ref{cond:CMFinVar}.
\end{enumeratea}

\subsection{Organization of the paper}
In Section \ref{s:proofs} we give the idea underlying the proof of Theorem \ref{thm-epidemic-curve} via the connection to first passage percolation. 
The intuitive idea is as follows. The expected proportion of vertices infected by time $t$ equals the probability that a {\em random} individual is infected by time $t$. Hence, we first prove a crucial proposition (Prop.\ \ref{prop:connection}) about the typical distance between two uniformly picked individuals in the graph, and then we perform first and second moment methods on the empirical proportion of infected individuals to obtain the epidemic curve. This approach first appeared in \cite{BarRei13}. We then explain the idea of the proof of Proposition \ref{prop:connection} in \cite{BhaHofHoo12} and a similar, implicitly given, result in \cite{BarRei13}: Both couple the initial phases of the infection to two branching processes, and describe how these clusters connect up. We explain how the connection happens based on the Bhamidi-van der Hofstad-Hooghiemstra (BHH) connection process, which proves that the process of possible connection edges converges to a Poisson process, of which the first point corresponds to the infection time. Essentially the same Poisson process appears in the connection process of \cite{BarRei13}, hence we just highlight the differences and similarities between these two approaches.

\section{Proofs}
\label{s:proofs}
In this section, we provide the proof of our main result Theorem \ref{thm-epidemic-curve}.
We start in Section \ref{sec:exploration-shortest-path} by describing the connection between 
exploration process on the configuration model and branching processes. Section \ref{ss:forwardbackward}
describes the relevant forward and backward continuous-time branching processes (CTBPs).
Section \ref{ss:labelling} provides the coupling between the infection process on the configuration model and the CTBPs. Section \ref{s:BP}, investigates asymptotics for the CTBPs.
Section \ref{ss:BHHconnect} describes how the forward and backward 
CTBP from two uniform vertices meet. Finally in Section \ref{ss:proof_main_theorem} these results are used 
to prove Theorem \ref{thm-epidemic-curve}. The intermediate Section \ref{ss:BRconnect}, we intuitively describe how asymptotics for the connection time is derived by Barbour and Reinert in 
\cite{BarRei13}.

\subsection{Exploration on the configuration model and branching processes}
\label{sec:exploration-shortest-path}

Consider the epidemic process $\cF_n(\cdot)$ with i.i.d.\ infection times and possibly infinite i.i.d.\ contagious period $C_v\in (0,\infty],~ v\in [n]$ with tail distribution $\bar H$. We shall see how this is connected to a shortest path problem on $\CMD$. To each directed edge $(v,u) \in \CMD$ assign an independent random edge length $X_{(v,u)}$ with distribution $G$. The epidemic process can be thought of as a flow starting at vertex $V_n$ at $t=0$ and spreading at rate one through the graph using the corresponding edge-lengths. When the infection hits a non-source vertex $v$ at time $\sigma_v$, thus infecting vertex $v$, each neighbor $u$ of $v$ (other than the neighbor that spread the infection to $v$) will be infected at time $\sigma_v+ X_{(v,u)}$ if $X_{(v,u)}$ is less than $C_v$. Thus the offspring distribution of new infections created by vertex $v$ -- describing the number of infections and infection times created by $v$ after $\sigma_v$  -- has the same distribution as 
	\begin{equation}
	\label{eqn:xi-def}
	\xi_v = \sum_{i=1}^{d_v-1} \delta_{X_i} \ind_{\set{X_i\le C_v}},
	\end{equation} 
where $d_v$ denotes the degree of $v$,  $X_i\sim G$ i.i.d.\ and $C_v\sim H$ is the contagious period of $v$.

\subsubsection*{Local neighborhoods in $\CMD$.} The initial source $V_n$ of the epidemic is picked uniformly at random from $[n]$ and thus has degree distribution $d_{V_n}$ in Condition \ref{cond:CMFinVar}. We next describe the neighborhood of this vertex.  By the definition of $\CMD$, we can construct $\CMD$ from $V_n$ by sequentially connecting the half-edges of $V_n$ to uniformly chosen unpaired half-edges. For any $j\geq 1$, let $N_j^*(n)\approx n p_j$ (by Condition \ref{cond:CMFinVar}) be the number of vertices  with degree $j$, where we exclude $V_n$. Then, for fixed $k\geq 1$, the probability that the first half-edge of $V_n$ connects to a vertex $v\in [n]\setminus \{V_n\}$ with degree $d_v =k+1$ equals   
	\begin{equation}
	\label{eqn:loc-nbhd}
	\frac{(k+1) N^*_{k+1}(n)}{\sum_{v\in [n]} d_v - 1} \approx \frac{(k+1)\prob(D_n=k+1)}{\E(D_n)}.
	\end{equation}
If $V_n$ connects to such a vertex, then this neighbor has $k$ remaining half-edges that can be used to connect to vertices in $\CMD$. Thus the forward degree of each neighbor $V_n$ has a distribution that is approximately equal to $D_n^\star$.  The same is true for the remaining half-edges of $V_n$ and, in fact,
the above approximation continues to hold as long as the neighborhood is not too large. Equation \eqref{eqn:xi-def} and \eqref{eqn:loc-nbhd} suggest that the epidemic process can be approximated by the following branching process $(\BP_n(t))_{t\geq 0}$ with label set $\BP_n(t)\subset \cV:=\set{0}\cup \cup_{n=1}^\infty \bN^n$.  
\begin{enumeratei}
	\item At time $t = 0$, start with a single individual $\rho=0$ whose offspring distribution is constructed as follows. First generate $d_{V_n}$ possible children and let $(X_i^0)_{1\leq i\leq d_{V_n}}$ be i.i.d.\ with distribution $G$ and independent of $C_0\sim (1-\bar H)$. Then the children of $\rho$ are the set $(0,i)$ such that $X_i < C_0$, labelled in an arbitrary order. The interpretation is that each of these vertices are born at time $X_i^0$. Thus the offspring distribution of the root can be represented as  
	\begin{equation}
	\label{eqn:xi-zero}
		\xi_0:= \sum_{i=1}^{d_{V_n}} \delta_{X_i} \ind_{\set{X_i \leq C_0}}.
	\end{equation}

	\item Every other individual $v\in \cV$ born into the process $\BP_n(\cdot)$, has i.i.d.\ offspring distribution $\xi_v$ with 
	\begin{equation}
	\label{eqn:xi-one}
		\xi_v:= \sum_{i=1}^{D_n^\star(v)} \delta_{X_i^v} \ind_{\set{X_i^v \leq C_v}},
	\end{equation}
where $C_v\sim (1-\bar H)$ is the contagious period,  $D_n^\star(v)$ has the size-biased distribution \eqref{eq:BiasedD} and $X_i^v$ i.i.d.\ $G$. Thus, conditionally on $D_n^\star(v)$, a vertex $(v,i)\in \cV$ is born at time $X_i^v$ after vertex $v$ is born if and only if $X_v^i \leq C_v$.  
\end{enumeratei} 

When $C=\infty$ this coupling between $\cF_n(\cdot)$ and the corresponding branching process $\BP_n(\cdot)$ is carried out in \cite[Section 4]{BhaHofHoo12}. The details and the corresponding error bounds turn out to be rather technical. We give an intuitive idea in Section \ref{ss:labelling} and Theorem \ref{xxx} gives a rigorous error bound for their difference.


\subsection{Forward and backward processes}\label{ss:forwardbackward}
In the previous section, we have described the branching process approximation to the epidemic \emph{forward} in time.  
Another key aspect of \cite{BarRei13} is the study of the \emph{backward} branching process. For a uniformly chosen vertex $w\in \CMD$ and fixed time $t>0$, the vertex $w$ is infected by time $t$ precisely when there is a chain of infections leading to $w$. Hence, for large time $t$, one can ask if $w$ is in the infection process of one of its neighbours, if that neighbour is in the infection process of one of his neighbours, etc, i.e., we can \emph{trace back the infection path}. In \cite{BarRei13}, this leads to a new approximating branching process, the \emph{backward} branching process with offspring process $\widetilde \xi[0,\infty]$. 

To see the difference between the offspring process $\xi$ going forward  and $\tilde \xi$ consider the case where all contagious periods are a.s.\ finite, i.i.d.\ having cumulative distribution function $H$.  Then, as before,  $\xi=\sum_{i=1}^{D^\star_n} \delta_{X_i} \indicator{X_i<C_v}$ denotes the offspring of the forward process. On the other hand, in the backward process each individual has to be \emph{in the contagious period of its children}, thus resulting in the offspring distribution 
	\be\label{eqn:xi-two}
	\widetilde \xi=\sum_{i=1}^{D^\star_n} \delta_{X_i} \indicator{X_i<C_i},
	\ee
where $C_i\sim H$ are i.i.d. In more complicated infection models 	the backward process turns out to be substantially more complicated to describe. 
The crucial observation is that in the case \eqref{eqn:xi-two}  $\Ev_n(\widetilde \xi(a,b))=\Ev_n(\xi(a,b))$ for all $0\le a<b \le \infty$ and thus the corresponding expected reproduction measure 
$\widetilde \mu_n(\mathrm dt)$ and $\mu_n(\mathrm dt)$ are the same for all $n$. This implies that when $C<\infty$, the distribution of  the limiting martingale variables defined in \eqref{eq:W} are not the same in the forward and backward processes, but the growth rate $\la_n$ and the multiplying constants for every characteristic under consideration, (see \eqref{eq:asconv}) are the same.

Note that if we take $C=\infty$, which is what we assume for the rest of the paper, the branching processes corresponding to the backward and forward processes are the same with offspring distribution 
	\begin{equation}
	\label{eqn:offspring-dist}
	\xi= \sum_{i=1}^{D_n^\star} \delta_{X_i}, \qquad X_i\sim G \mbox{ are i.i.d.\ random variables.}
	\end{equation}
From now on every quantity $\widetilde Q$ corresponds to the quantity $Q$ in the backward process.
\subsection{Labeling the BP with half-edges on the configuration model}\label{ss:labelling} 
We now construct $\CMD$ along with the epidemic process $\cF_n(\cdot)$ on it.  
First we construct the forward process by describing the sequence of new vertices that are infected and the times that these vertices get infected. At each step $k\geq 0$, one of two things can happen:
\begin{enumeratei}
\item {\bf Event I:} A new vertex gets infected via an active half-edge from the set of currently infected vertices connecting to a half-edge in the set of susceptible vertices. The rest of the half-edges connected to this newly infected vertex are now designated to have joined the active half-edges, while the two half-edges that merge to create this connection are removed. 
\item {\bf Event II: } Occasionally two active half-edges in the infected cluster merge to create a new edge. The number of times this happens before time $t_n$ is a tight random variable. This obviously does not increment the infected cluster since a new vertex is not added to the cluster. 
\end{enumeratei} 

We now give a precise description of the construction. Let $\CL_n$ denote the set of half-edges in $\CMD$. For $x\in \cL_n$, let $V(x)\in [n]$ denote the vertex it is attached to, $p_x$ denote the half-edge it is merged to and let $V(p_x) $ denote the vertex incident to this half-edge. 
 
For $k=0$, pick the source of the infection $V_n\in[n]$ uniformly at random. This vertex has $d_{V_n}$ offspring and is born immediately.  Set $\tau_0=0, \CF'(0)=\{V_n\}$. Check if any of these half-edges are merged amongst themselves creating self-loops:
this happens with probability $o(1)$. The half-edges where this does not take place form the coming generation $\CA_{\tau_0}$ with residual times to birth given by  $\CB_{\tau_0}:=(B_x(\tau_0))_{x\in \CA_{\tau_0}}$ with $B_x(\tau_0)\sim G$ i.i.d. 
For each $x\in \cA_{\tau_0}$, the end point $V(p_x)$ is revealed and infected at time $X_x$.   
Write $\CH_{\tau_0} = \CL_n \setminus \set{x:V(x)=V_n}$ for the initial set of free half-edges. 

For $k\ge1$, the construction proceeds recursively as follows. At this stage, we have the set of active half-edges $\cA_{\tau_{k-1}}$ and free half-edges $\CH_{\tau_{k-1}}$, as well as residual times of birth of the active half-edges $\CB_{\tau_{k-1}}$. 
\begin{enumeratea}
\item Pick half-edge $x_k^\star$ with shortest residual time to birth: $B_k^\star=\min \CB_{\tau_{k-1}}$ and pair it to a uniformly chosen free half-edge $p_{x_k^\star}\in \CH_{\tau_{k-1}} \cup \CA_{\tau_{k-1}}$. 
Update time $\tau_k:=\tau_{k-1}+B_k^\star$.

\item Add the vertex $v_k:=V(p_{x_k^\star})$ to the infected vertices $\CF'(\tau_k)$. 
Check all other half-edges of $v_k$ (other than $p_{x_k^\star}$) to see if any of them are attached to one of the other active half-edges in $\cA_{\tau_{k-1}}$ and let $\cV_k^\star$ denote the residual set of half edges of $v_k$. More precisely, we draw a Bernoulli variable with success probability equal to the number of active half-edges over the total number of unpaired half-edges. If the Bernoulli equals 1, then we pair the half-edge to a uniform active half-edge, if it equals 0, then
we do not yet pair it.

\item Refresh the coming generation: The new set of active half-edges is defined as  
	\[
	\CA_{\tau_k} := \CA_{\tau_{k-1}} \cup \cV_k^\star \setminus \set{x_k^\star, p_{x_k^\star}}.
	\] 

\item Refresh residual times to birth 
	\[
	\CB_{\tau_k}:=\set{B_x(\tau_{k-1})-B_k^\star\colon x\in \CB_{\tau_{k-1}} \setminus\set{x_k^\star}} 
	\bigcup \set{X_y \colon y\in \cV_k^\star}, 
	\] 
i.e., we remove $B_k^\star$ from all residual times to birth and add the i.i.d.\ edge weights $X_y$ for newly active half-edges.

\item We refresh the free half-edge-set: $\CH_{\tau_{k}} := \CH_{\tau_{k-1}} \setminus \set{x: V(x) = v_k}$, that is, we remove the half-edges of $v_k$. 
\end{enumeratea}

Let $(\cF_n'(k))_{k\geq 0}$ denote the above discrete-time process.  By construction, the following lemma is obvious:

\begin{Lemma}
\label{lem:exploration-forward}
For any $t> 0$, set $k(t) = \sup\set{k\colon \tau_k \leq t}$. Let $\cF_n^*(t) := \cF_n'(k(t)) $. Then, for the epidemic process on $\CMD$, the distributional equality $(\cF_n(t))_{t\geq 0} \stackrel{d}{=} (\cF_n^*(t))_{t\geq 0}$  holds.
\end{Lemma}


\noindent{\bf Coupling to a branching process.} 
In \cite[Section 4]{BhaHofHoo12} it is shown that the above construction of the epidemic process can be coupled to a branching process $\BP_n(\cdot)$ where the root has offspring distribution \eqref{eqn:xi-zero} and all other individuals have distribution \eqref{eqn:xi-one} (both with $C_v = \infty$). The intuitive idea is as follows: for the two events above; Events I correspond to creation of new vertices both in $\cF_n$ and $\BP_n$ while Events II correspond to the creation of {\bf artificial vertices} in $\BP_n$. Now let $\BP$ denote the ($n$-independent) branching process where the offspring distributions in \eqref{eqn:xi-zero}, \eqref{eqn:xi-one}, we replace $d_{V_n}, D_n^\star$ by their distributional limits $D, D^\star$.   Let $d_{\sss{\mathrm TV}}(\cdot, \cdot)$ denote total variation distance between these mass functions on $\bN$. Define $t_n$, $s_n\to\infty$ with $\set{s_n}_{n\geq 1}$ being a sequence satisfying
 \begin{equation}
 \label{eqn:tn-sn}
 	t_n = \log n/\la_n, \qquad \e^{\lambda s_n} d_{\sss {\rm TV}}(D_n^\star, D^\star)\to 0.
 \end{equation} 
\begin{Proposition}[{\cite[Prop 2.4]{BhaHofHoo12}}]
\label{xxx}
There exists a coupling of the processes $(\cF_n(t))_{0\leq t\leq s_n}$ and $(\BP(t))_{0\leq t\leq s_n}$ such that 
	\[
	\prob\Big((\cF_n(t))_{0\leq t\leq s_n} \neq (\BP(t))_{0\leq t\leq s_n}\Big)\to 0 \qquad \mbox{as } n\to\infty.
	\]
Further, there exists a coupling between $\cF_n$ and $\BP_n$ such that the above bound holds with $\BP$ replaced with $\BP_n$.
\end{Proposition}

\subsection*{Exploration of the backward infection process} 
After time $t_n^\star\approx\frac{1}{2\la_n} \log n$ specified later, we freeze the forward cluster. The `half-edges sticking out' of this cluster namely the set of active edges are exactly the ones in the coming generation $\CA_{t_n^\star}$. We start labelling the backward process \emph{conditional on the presence of the forward process}.
This labelling is slightly different than the labelling of the forward cluster, since we also want to keep track when we connect to a half-edge in the coming generation $\CA_{t_n^\star}$. 

At each step $k\geq 0$, three things can happen in the backward process: Event I and II defined above in the forward process or
\begin{enumeratei}
\item[(iii)] {\bf Event III:} Occasionally we pair a half-edge in the backward cluster to a half-edge in the coming generation of the forward cluster $\CA_{t_n^\star}$. This  means that a \emph{collision} happens between the two processes. 
\end{enumeratei} 
We now give a precise description of the construction.

For $k=0$, pick the source of the backward-infection $\widetilde V_n\in[n]\setminus \{ \mathrm{\CF}_n({t_n^\star}) \}$ uniformly. This vertex has $d_{\widetilde V_n}$ offspring and is born immediately.  Set $\widetilde \tau_0=0$. 
Pair the $d_{\widetilde V_n}$ outgoing half-edges immediately, uniformly at random without replacement from $\CA_{t_n^\star} \cup \CH_{t_n^\star}$. Check if any of these half-edges are merged amongst themselves creating self-loops (Event II) or collision edges (Event III). Set the collision edges and residual collision times and the coming generation or \emph{active edges} for Event I by  
	\[ 
	\ba \CC_0&:= \{ ((y, p_y), B_{p_y}(t_n^\star)): V(y)=\widetilde V_n,  p_y \in \CA_{t_n^\star}\},\\
	\widetilde \CA_0&:= \{(y,p_y): V(y)=\widetilde V_n, p_y \notin \CA_{t_n^\star}, V(p_y) \neq \widetilde V_n\}.
	\ea 
	\] 
For Event III: if there is a $(y, p_y)$ with $p_y \in \CA_{t_n^\star}$ forms an edge between $\widetilde V_n$ and the forward cluster. From this edge there is already some time 'eaten up' by the forward cluster: the remaining time on this edge is $B_{p_y}(t_n^\star)$. Remove Event II pairs $(y, p_y)$ from the set of active edges: they form a self-loop. For Events I, the initial remaining times to birth $\widetilde \CB_{0}:=\{B_x(\widetilde \tau_0), x\in \widetilde \CA_{0}\}$ with $B_x(\widetilde \tau_0)\sim G$ i.i.d. For each $y\in \widetilde \cA_{0}$, the end point $V(p_y)$ is \emph{revealed immediately but infected only} at time $X_y$.   The initial set of free half-edges is
	\[
	\widetilde \CH_{0} = \left(\CA_{t_n^\star} \cup \CH_{t_n^\star}\right)  \setminus 
	\left(\{y: V(y) = \widetilde V_n\} \cup \{ p_y: V(y)=\widetilde V_n\}\right).
	\] 
In more detail, we remove from $\CA_{t_n^\star} \cup \CH_{t_n^\star}$ the half-edges of $\widetilde V_n$ and their pairs. For $k\ge1$ the construction proceeds as follows. At this stage we have the set of active edges $\widetilde \cA_{\widetilde \tau_{k-1}}$ and free half-edges $\widetilde\CH_{\widetilde \tau_{k-1}}$ as well as residual times of birth of the active edges $\widetilde \CB_{\widetilde \tau_{k-1}}$. This is described in the following process:

\begin{enumeratea}
\item Pick an active edge $(\widetilde x_k^\star, p_{\widetilde x_k^\star}) \in \widetilde \CA_{\widetilde \tau_{k-1}}$ with shortest residual time to birth: $\widetilde B_k^\star=\min \widetilde \CB_{\widetilde \tau_{k-1}}$.  
\item Set the time $\widetilde \tau_k:=\widetilde \tau_{k-1}+\widetilde B_k^\star$.
\item Add the vertex $\widetilde v_k:=V(p_{\widetilde x_k^\star})$ to the infected vertices $\widetilde \CF(\widetilde \tau_k)$ 
\item refresh the coming generation and the collision edges: pair all half-edges $y: V(y)=\widetilde v_k$ sequentially to a uniformly chosen half-edge $p_{y}\in \widetilde \CH_{\widetilde \tau_{k-1}} \cup \widetilde \CA_{\tau_{k-1}}$.

The new set of collision and active edges is defined as  
	\[ 
	\ba 
	\CC_{\widetilde \tau_k}&:= \CC_{\widetilde \tau_{k-1}} \cup 
	\{ ((y, p_y), B_{p_y}(t_n^\star) ): V(y)=\widetilde v_k, p_y \in \CA_{t_n^\star} \},
	\\
	\widetilde \CA_{\widetilde \tau_k} 
	&:= \widetilde \CA_{\tau_k} \cup \set{(y, p_y): V(y) = \widetilde v_k, p_y \notin \CA_{t_n^\star}\cup \widetilde \CA_{\widetilde \tau_{k-1}}} 
	\setminus \set{x_k^\star, p_{x_k^\star}},
	\ea
	\] 
namely, the new collision edges are those among the $d_{\widetilde v_k}-1$ newly found half-edges whose pair is an active half-edge in the forward process, and the remaining time on this edge is $B_{p_y}(t_n^\star)$. If $p_y\in \widetilde \CA_{\tau_{k-1}}$, then Event II happens: we have found a cycle. If none of this is the case, then the edge $(y,p_y)$ becomes an active edge with residual time to birth $B_y=X_{y}\sim G$ i.i.d.
\item Refresh the residual times to birth 
	\[
	\widetilde \CB_{\tau_k}:=
	\set{ B_x(\widetilde\tau_{k-1})-\widetilde B_k^\star
	\colon x\in \widetilde \CB_{\widetilde \tau_{k-1}} 
	\setminus\set{\widetilde x_k^\star}} \bigcup 
	\set{X_y: V(y) =\widetilde v_k, y\neq p_{\widetilde x_k^\star}, p_y \notin \CA_{t_n^\star}\cup \widetilde \CA_{\widetilde \tau_{k-1}}}.
	\] 
That is, we subtract $B_k^\star$ from all residual times to birth and add the i.i.d.\ edge weights $X_y$ for newly active edges (but we do not add the remaining time of collision edges and we remove cycle-edges too).
\item Refresh the free half-edge-set: $\widetilde \CH_{\tau_{k}} = \widetilde \CH_{\tau_{k-1}} \setminus (\set{y: V(y) = \widetilde v_k} \cup\set{ p_y: V(y)=\widetilde v_k})$, namely remove the half-edges of $\widetilde v_k$ and their pairs. 
\end{enumeratea}

The main difference of this process and the forward process is that here we pair the new outgoing half-edges $y\in \{1, \dots, d_{\widetilde v_k}-1\}$ immediately at the birth of $\widetilde v_k$, and we check if this edge collides with the forward cluster or becomes active. (Hence in the backward process, the pairs $(x, p_x)$ form the coming generation.)
The statement of Proposition \ref{xxx} remains valid for this process as well, i.e. the coupling between the backward cluster and $\BP$ can be established.

\noindent {\bf The total length of collisions.}
A \emph{collision} happens at time $\widetilde \tau_k$ for some $k$ if the vertex $\widetilde v_k$  has a half-edge $y$ with a pair $p_y\in \CA_{t_n^\star}$ of the forward process. Since we check this exactly at the time when $\widetilde v_k$ becomes infected, and there is still a residual time $B_{p_y}(t_n^\star)$ on this edge, the length of this connection is exactly $t_n^\star + B_{p_y}(t_n^\star) + \widetilde \tau_k$.

Note that $p_y$ is a uniformly picked half-edge from the coming generation $\CA_{t_n^\star}$, hence its residual time to birth $B_{p_y^\star}(t_n^\star)$ converges to the empirical residual time to birth distribution in \eqref{eq:remainings} below. Also note that this is independent of the backward process infection time $\widetilde \tau_k$.

\subsection{Branching processes}
\label{s:BP}
In this section we set up the branching process objects including the stable-age distribution theory \cite{Ner81} required to prove the result.  Fix a point process $\xi$ on $\bR_+$ and consider a branching process $\BP(\cdot)$ with vertex set a subset of $\cN:=\set{0}\cup \cup_{n=1}^\infty \bN^n$, started with one individual $0$ at $t=0$ with each vertex having an i.i.d. copy of $\xi$. Here an individual is labeled $x=(i_1i_2 ,\dots,i_n)$ if $x$ is the $i_n$th child of the $i_{n-1}$th child of $\ldots$ of the $i_1$th child of the root. For $t\geq 0$, let $\xi[t]$ denotes the number of points in $[0,t]$. Write $\mu(t)=\Ev[\xi(t)]$ for the corresponding intensity measure. Assume $\mu(\cdot)$ is non-lattice, there exists a Malthusian parameter $\la\in(0,\infty)$ satisfying  
	\be
	\label{eq:maltusian} 
	\int_0^{\infty} \e^{-\la t} \mu(\mathrm d t) =1, 
	\ee
and with integrability assumptions for this parameter $\lambda$, 
	\begin{equation}
	\label{eq:meanage-xlogx}
	m^\star:=\int_0^\infty t \e^{-\la t} \mu(\mathrm d t) < \infty, \qquad  
	\Ev\left( \int_0^\infty \e^{-\la t } \xi(\dd t) \cdot \log^+ 
	\left(\int_0^\infty \e^{-\la t } \xi(\dd t) \right) \right)<\infty.	
	\end{equation}
 %
 
For $v\in \BP$, write $\sigma_v$ for its birth time and $\xi_v$ for its offspring process. Let $\set{\set{\phi_v(\cdot)}: v\in \BP}$ be a family of i.i.d.\ stochastic processes with $\set{\phi_v(t): t\geq 0}$ measurable with respect to the offspring distribution $\xi_v$, $\phi_v(t)\geq 0$ for $t\geq 0$ and let $\phi_v(s) = 0$ for $s<0$. The interpretation of such a functional, often called a \emph{characteristic} \cite{Jage75,Ner81,JagNer84} is that it assigns a score $\phi_v(t)$ when vertex $v$ has age $t$. We write $\phi:=\phi_0$ to denote this process for the root.   
 %
The branching process counted according to this characteristic is defined as
	\[ 
	Z_t^\phi:= \sum_{x\in \mathcal \BP(t)} \phi_x(t-\sigma_x).
	\]
Theorem 5.4 and Corollary 5.6 in \cite{Ner81} shows that there exists a random variable 
$W\ge 0$ with $\Ev[W]=1$ such that for any characteristic $\phi$ satisfying mild integrability conditions one has 
   	\be
	\label{eq:asconv} 	
	\e^{-\la t} Z_t^{\phi} \longrightarrow W \cdot
	\frac{\int_0^\infty \e^{-\la t} \Ev(\phi(t)) \mathrm dt }{m^\star}\quad \mbox{a.s. }
	\ee
Moreover, for two characteristics $\phi_1$ and $\phi_2$ we have
	\be 
	\label{eq:ratioconv} 
	\frac{Z_t^{\phi_2}}{Z_t^{\phi_1}} 
	\longrightarrow \frac{\int_0^\infty \e^{-\la t} \Ev(\phi_2(t)) \mathrm dt }
	{\int_0^\infty \e^{-\la t} \Ev(\phi_1(t)) \mathrm dt } \quad \mbox {a.s. on } \{W>0\}.
	\ee

Now we apply this general theory for our epidemic - exploration process on $\CMD$.  We fix $n$ first. Recall that the epidemic process $\cF_n(\cdot)$ on $\CMD$ is approximated by a branching process $\BP_n$ with offspring process $\xi= \sum_{i=1}^{D^\star_n} \delta_{X_i}$. There is a slight modification for the distribution of the root, however this does not effect the limit theorems above (other than the limit random variable having $\E(W_n) \neq 1$). Recall the Malthusian rate of growth parameter $\lambda_n$ from \eqref{eqn:mainparameters-lamb}. The other parameters (with $n$ fixed) are calculated as
	\begin{equation}
 	\label{eq:mainparameters}
 	\mu_n(t) := \Ev(D^\star_n) \int_0^t \bar H(x) G(\dd x), \qquad 
	\mu_n(\mathrm dt): =\Ev(D^\star_n) \bar H(t)G(\mathrm dt), \qquad 	
	m^\star_n =\Ev(D^\star_n) \int_0^\infty t \e^{-\la_n t} \bar H(t)G(\mathrm d t).
 	\end{equation}

The parameter $m_n^\star$ is called the \emph{mean of the stable age distribution} or mean age at childbearing.
In order to establish the connection between two infected clusters in the graph, we shall need the size of the so called \emph{coming generation} (i.e., those individuals who will be born after time $t$ but their mother was born before time $t$),
and the empirical distribution of the \emph{residual time to birth} of a uniformly picked individual in the coming generation. Asymptotics for these objects are derived by choosing appropriate characteristics. Fix $s>0$. If we set $\phi^s(t):=\xi[t+s,\infty)$ then $Z_t^{\phi^s}=\sum_{x\in \CF} \xi_x[t-\sigma_x+s,\infty]$ counts the number of children of already born individuals whose birth date is at least $s$ time units from now.
In particular,  we write $A_t^d:=Z_t^{\phi_0}=\sum_{x\in \CF} \xi_x[t-\sigma_x,\infty]$ counting the size of the coming generation (usually referred to as \emph{alive} individuals in CTBP literature) in a BP with reproduction measure $\mu_n$ in \eqref{eq:mainparameters}.  (We add the superscript $d$ for \emph{delaying} the process by one generation, i.e.\  the root here has also $\mu_n$)
We calculate using \eqref{eq:mainparameters} that in our case $\Ev(\phi_0)=\Ev(D^\star_n)\cdot \int_t^\infty \bar H(x) G(\mathrm d x)$ hence:
	\be
	\ba
	\label{eq:asalivesdelayed} 
	\e^{-\la_n t}A_t^d&= \e^{-\la_n t}Z_t^{\phi_0} \longrightarrow  W^d_n \cdot 
	\frac{\int_0^\infty \e^{-\la_n t} \Ev(D^\star_n)\int_t^\infty \bar H(x)G(\mathrm d x) \mathrm dt}
	{m^\star_n}\\
	&=W^d_n \cdot \frac{\Ev(D^\star_n) \int_0^\infty \bar H(x)G(\mathrm dx)-1}{m^\star_n \la_n}= W_n^d\cdot \frac{\mu_n(\infty)-1}{m_n^\star \la_n} \quad \mbox{ a.s.}
	\ea
	\ee
Now, to match the $\BP$ to the exploration process $\CF_n(t)$ on $\CMD$ to have the same reproduction function at the root, we introduce the following $BP$ via the size of the coming generation by 	
	\[A_t:=\sum_{i=1}^{D_n} \left(\indicator{t<X_i<C_v}+A _{t-X_i}^{d,(i)}
	\indicator{X_i< t\wedge C_v}\right), 
	\]
where $A^{d,(i)}_t$ are i.i.d. copies of $A_t^d$ in \eqref{eq:asalivesdelayed}. $A_t$ corresponds to $|\CA_n(t)|$, i.e. the number of active half-edges in $\CF_n(t)$. Multiplying by $e^{-\la_n t}$ and using \eqref{eq:asalivesdelayed} gives the convergence
	\be
	\ba\label{eq:alives} \e^{-\la_n t} A_t&= \e^{-\la_n t}\sum_{i=1}^{D_n} \indicator{t<X_i<C_v} +\sum_{i=1}^{D_n} \e^{-\la_n X_i}\indicator{X_i<t\wedge C_v} \left( \e^{-\la_n (t-X_i) } A_{t-X_i}^{d,(i)}\right) \\
&\toas \sum_{i=1}^{D_n} \e^{-\la_n X_i} \indicator{X_i<C_v}W^{d,(i)}_n \frac{\mu_n(\infty)-1}{\la_n m^\star_n}, \ea\ee
with $W^{d,(i)}_n$ i.i.d. copies of $W_n^d$. Since $\Ev(\e^{-\la_n X_i}\indicator{X_i<C_v}) = \frac{1}{\Ev(D^\star_n)}$ by \eqref{eqn:mainparameters-lamb}, and $X_i$ is independent of $W^{d,(i)}_n$, we can introduce the limiting random variable $W_n$ in \eqref{eqn:mainparameters-lamb}:
	\be
	\label{eq:W} 
	W_n := \sum_{i=1}^{D_n} \e^{-\la_n X_i} \indicator{X_i<C_v}W^{d,(i)}\frac{\mu_n(\infty)-1}{\la_n m^\star_n} , 
	\ee
and then \eqref{eq:alives} implies
	\be
	\label{eq:asalives} \e^{-\la_n t} A_t \toas 
	W_n \quad \mbox{ with } \quad \Ev[W_n]= \frac{\Ev[d_{V_n}](\mu_n(\infty)-1)}{\Ev[D^\star_n]\la_n m^\star_n}. 
	\ee
For infinite contagious period we have $\mu_n(\infty)-1 = \Ev[D^\star_n-1]$.
The ratio convergence in \eqref{eq:ratioconv} and $\Ev(\phi_s)=\Ev(D^\star_n)\cdot \int_{t+s}^\infty \bar H(x) G(\dd x)$ implies that the empirical `residual time to birth' distribution converges to a random variable:
	\be
	\label{eq:remainings} 
	\frac{Z_t^{\phi_s}}{Z_t^{\phi_0}}\toas \frac{\Ev(D^\star_n)\int_0^\infty 
	\e^{-\la_n t} \int_{t+s}^\infty \bar H(x) G(\dd x) \mathrm dt}
	{\Ev(D^\star_n)\int_0^\infty \e^{-\la_n t} \int_{t}^\infty \bar H(x) G(\dd x)\mathrm dt} 
	= \frac{\Ev(D^\star_n)}{\mu(\infty)-1} \int_s^\infty (1-\e^{\la_n (s-x)} )
	\bar H(x)G(\mathrm d x):=1-F_{\sss R}^{\sss{(n)}}(s).
	\ee
This is the limiting probability that a uniformly picked individual from the `coming generation' will be born after an extra $s$ time units. 

Now we have set the stage for the branching processes that approximate the initial phase of the infection and the backward infection process. We are ready to state the main proposition on which our proof of the epidemic curve is based. Let us denote the infection time from $v$ to $w$ by $L_n(v,w)$. (The first part of this proposition is part of Theorem 1.2 in \cite{BhaHofHoo12}, the second is a two-vertex analogue of it that can be proved in a similar way.) In its statement, and for $s>0$, we let $\cG_n(s)$
denote the $\sigma$-algebra of all vertices that are infected before time $s$, as well as all edge weights of the half-edges that are incident to such vertices. Thus, as opposed to $\cF_n(s)$ which has information only about the sequence of transmissions that have transpired before time $s$, $\cG_n(s)$ also contains information about the ``coming generation'' of infections.

\begin{Proposition}\label{prop:connection} 
Take $s_n$ as in Proposition \ref{xxx}.
The shortest infection path between two uniformly picked vertices $V_n$ and $\widetilde V_n$ satisfies
	\be
	\label{eq:shortest-path-prob}	
	\Pv\left( L_n(V_n, \widetilde V_n) -\frac{\log n}{\la_n} 
	+ \frac{\log W_{s_n}}{\la_n} 
	+ \frac{\log \widetilde W_{s_n}}{\la_n}<t ~\Big|~\cG_n(s_n), \widetilde \cG_n(s_n) \right) 
	\toindis \Pv( -\La/\la +c < t ).
	\ee
Further, with $V_n$, $\widetilde V_n^{\sss(1)}$ and $\widetilde V_n^{\sss(2)}$ three independent uniform
vertices in $[n]$, and their forward and backward infection processes $\cG_n(s_n), \widetilde \cG_n^{\sss(1)}(s_n), \widetilde \cG_n^{\sss(2)}(s_n)$,
	\eqan{
	\label{eq:shortest-path-prob-two}	
	&\nn\Pv\Big( L_n(V_n, \widetilde V_n^{\sss(i)}) -\frac{\log n}{\la_n} 
	+ \frac{\log W_{s_n}}{\la_n} 
	+ \frac{\log \widetilde W_{s_n}^{\sss(i)}}{\la_n}<t,~i=1,2
	~\Big|~ \cG_n(s_n), \widetilde \cG_n^{\sss(1)}(s_n),
	\widetilde \cG_n^{\sss(2)}(s_n) \Big)\nn\\
	&\qquad \qquad
	\toindis \Pv( -\La/\la +c < t )^2.
	}
\end{Proposition}

\subsection{The Bhamidi-van der Hofstad-Hooghiemstra connection process}
\label{ss:BHHconnect}
In this section, we describe the results on the connection process in \cite{BhaHofHoo12}.
We start by setting the stage. Fix the deterministic sequence $s_n\rightarrow \infty$ as in Proposition \ref{xxx}. Then, define
    	\eqn{
    	\label{tn-def}
    	t_n=\frac{1}{2\alphan} \log{n},
	\qquad
	\tn=\frac{1}{2\alphan} \log{n}
	-\frac{1}{2\alphan} \log{\big(W_{s_n}\widetilde{W}_{s_n}\big)}.
    	}
Note that $\e^{\alphan t_n}=\sqrt{n}$, so that at time $t_n$, both $\cF_n(t_n), \widetilde \cF_n(t_n)$
have size of order $\sqrt{n}$; consequently the variable $t_n$ denotes the typical time when collision edges start appearing. The time $\tn$ incorporates for stochastic fluctuations in the size of these infected  (and backward-infected) clusters. 

By Proposition \ref{xxx}, $s_n\rightarrow \infty$ is such that $\cF_n(s_n)$ and $\widetilde \cF_n(s_n) $ for $t\leq  s_n$ can be coupled with two independent CTBPs. For the present part, it is crucial that the forward CTBP from $V_n$ and the backward CTBP from $\widetilde V_n$ are run \emph{simultaneously}. That is, we run the two exploration processes described in Section \ref{ss:labelling} at the same time. 

We say that an edge is a {\em collision edge} when, upon pairing it, it connects to a half-edge in the 
other CTBP, i.e., either a half-edge in the coming generation of the forward cluster of $V_n$ pairs to a half-edge 
in the coming generation of the backward cluster of $\widetilde V_n$, or the other way around. The main result in 
this section describes the limiting stochastic process of the appearance of the collision edges, as 
well as their properties. In order to do so, we introduce some more notation.

Denote the $i$th collision edge by $(x_i,p_{x_i})$, where $p_{x_i}$ 
is an active half-edge (either in the forward or in the backward cluster) and $x_i$ the half-edge which pairs to $p_{x_i}$.
Further, let $\Tcol_i$ denote the time at which the $i$th collision edge is formed, which is the same as 
the birth time of the vertex incident to $x_i$. We let $R_{\Tcol_i}(p_{x_i})$ be the remaining life time of
the half-edge $p_{x_i}$, which, by construction is equal to the time after time $2\Tcol_i$ that the edge will
be found completely by the flow. Thus, the path that the edge $(x_i,p_{x_i})$ completes has length equal to
$2\Tcol_i+R_{\Tcol_i}(p_{x_i})$ and it has $H(x_i)+H(p_{x_i})+1$ edges, where $H(x_i)$ and $H(p_{x_i})$ 
denote the number of edges between the respective roots and the vertices incident to $x_i$ and $p_{x_i}$,
respectively. We conclude that 
the shortest weight path has weight equal to
$L_n(V_n\widetilde V_n)=\min_{i\geq 1} [2\Tcol_i+R_{\Tcol_i}(p_{x_i})]$. Let $J$ be the minimizer of this minimization problem. Then, the number of edges is equal 
to $H_n=H(x_J)+H(p_{x_J})+1$. Finally, for a collision edge $(x_i,p_{x_i})$, we let $I(x_i)=1$ when
$x_i$ is incident to a vertex that is part of $\CF_n(\Tcol_i)$ and $I(x_i)=2$ when
$x_i$ is incident to a vertex that is part of $\widetilde\CA_n(\Tcol_i)$.

In order to describe the properties of the shortest weight path, we define
    \eqn{
    \label{eq:resc-variables}
    \barTcol_i=\Tcol_i-\tn,
    \qquad
    \bar{H}_{i}^{\sss({\rm or})}=\frac{H(x_i)- t_n/m_n^{\star}}{\sqrt{(\sigma_n^{\star})^2t_n/(m_n^{\star})^3}},
    \qquad
    \bar{H}_{i}^{\sss({\rm de})}=\frac{H(p_{x_i})- t_n/m_n^{\star}}{\sqrt{(\sigma_n^{\star})^2t_n/(m_n^{\star})^3}},
    }
where $m_n^{\star}$ is the mean 
of the stable-age distribution in \eqref{eq:mainparameters}, while $\sigma_n^{\star}$ is its standard 
deviation.

We write the random variables $(\Xi_i)_{i\geq 1}$ with $\Xi_i\in \R\times \{1,2\}\times \R\times \R\times [0,\infty),$
by
    \eqn{
    \label{Ui-def}
    \Xi_i=\big(\barTcol_i, I(x_i), \bar{H}_{i}^{\sss({\rm or})}, \bar{H}_{i}^{\sss({\rm de})}, R_{\Tcol_i}(p_{x_i})
    \big).
    }
Then, for sets $A$ in the Borel $\sigma-$algebra of the space $\Scal:= \R\times \{1,2\}\times \R\times \R\times [0,\infty)$,
we define the point process
    \eqn{
    \label{PPP-discr}
    \Pi_n(A)=\sum_{i\geq 1} \delta_{\Xi_i}(A),
    }
where $\delta_x$ gives measure $1$ to the point $x$. Let $\Mcal(\Scal)$ denote the space of all simple locally finite point processes on $\Scal$ equipped with the vague topology (see e.g.\ \cite{KallenBK76}). 
 On this space one can naturally define the notion of  weak convergence of a sequence of random point processes $\Pi_n\in \Mcal(\Scal)$. This is the notion of convergence referred to in the following theorem. In the theorem, we let $\Phi$ denote the distribution function of a standard normal
random variable. Finally, we define the density $\fR$ of the limiting \emph{residual time to birth distribution} $F_R$ in \eqref{eq:remainings} given by
	\eqn{
	\label{dens_fR}
	\fR(x)=\frac{\int_0^\infty \e^{-\lambda y}g(x+y)\,dy}{\int_0^\infty \e^{-\lambda y} [1-G(y)]\,dy}.
	}
Then, our main result about the appearance of collision edges is the following theorem:

\begin{Theorem}[PPP limit of collision edges]
\label{thm-main-PPP}
Consider the distribution of the point process $\Pi_n \in \Mcal(\Scal)$ defined in
\eqref{PPP-discr} conditional on $((\cF_n(s), \widetilde \cF_n(s)))_{s\in[0,s_n]}$ such that $W_{s_n}>0$ and $\widetilde W_{s_n}>0$. 
Then $\Pi_n$  converges in distribution as $n\to\infty$ to a Poisson Point Process (PPP) $\Pi$ 
with intensity measure
    \eqn{
    \label{intensity-PPP}
    \lambda(dt\times i\times dx\times dy\times dr)
    =\frac{2
	\Ev[D^\star] \fR(0)}{\expec[D]}\e^{2\lambda t}dt
    \otimes
    \{1/2,1/2\}
    \otimes
    \Phi(dx)
    \otimes
    \Phi(dy)
    \otimes
    \FR(dr).
    }
\end{Theorem}

Write the points in the above PPP as $(P_i)_{i\geq 1}$. In \cite{BhaHofHoo12}, it is shown that
Theorem \ref{thm-main-PPP} implies that $L_n(V_n,\widetilde V_n)-2\tn\convd \min_{i\geq 1}[2P_i+R_i]$.
Further, it follows that
	\eqn{
    	\label{rewrite-Gumbel}
    	\min_{i\geq 1}(2P_i+R_i)
	\stackrel{d}{=}-\Lambda/\lambda-\log(
	\Ev[D^\star]\fR(0) B/\expec[D])/\lambda,
    	}
with $B=\int_{0}^{\infty} \FR(z) \e^{-\lambda z}\,dz=m^\star/\Ev[D^\star-1]$, where $m^\star$ is the mean of the so-called {\em stable-age distribution} in \eqref{eq:mainparameters}. 
In \cite[Lemma 2.3]{BhaHofHoo12}, it is shown that $\fR(0)=\lambda/\expec[D^{\star}-1]$, 
so that
$c=-\log(\Ev[D^\star]\fR(0)B/\expec[D])/\la=\log(\expec[D]\expec[D^\star-1]^2/(\lambda\expec[D^{\star}]m^\star))/\la$.

Here we thus see that the Gumbel distribution arises from the minimization of
the points of the PPP $(2P_i+R_i)_{i\geq 1}$. Interestingly, the Gumbel distribution {\em also} arises 
in $\min_{i\geq 1} P_i$, but with a different constant $c$. Thus, the addition of the residual life-time 
only changes the constant. Since $2(\tn-t_n)\convd \frac{1}{\la} \log{\big(W\widetilde{W}\big)}$,
this proves that 
	\eqn{
	\label{Ln-lim-pf}
	L_n(V_n,\widetilde V_n)-2t_n=\min_{i\geq 1} [2\Tcol_i+R_{\Tcol_i}(p_{x_i})]-2t_n
	\convd -\Lambda/\lambda+c +\frac{1}{\la} \log{\big(W\widetilde{W}\big)}.
	}
Also, by \refeq{resc-variables}, 
the trail of the epidemic, which is equal to $H_n=H(x_J)+H(p_{x_J})+1$, satisfies that
$(H_n-2t_n/m_n^{\star})/\sqrt{(\sigma_n^{\star})^2t_n/(m_n^{\star})^3}$ converges in distribution to
the sum of two i.i.d.\ standard normal random variables, where $m_n^{\star}$ is the mean 
of the stable-age distribution in \eqref{eq:mainparameters}, while $\sigma_n^{\star}$ is its standard 
deviation. This explains \eqref{Pn(t,h)-conv}, and identifies $\alpha_n=1/(\alphan m_n^{\star})$ and
$\beta=(\sigma^{\star})^2/[\lambda (m^{\star})^3]$, where $m^{\star}=\lim_{n\rightarrow \infty} m_n^{\star}$ and $\sigma^{\star}=\lim_{n\rightarrow \infty} \sigma_n^{\star}$.

To prove Theorem \ref{thm-main-PPP}, we investigate the expected number of collision edges that are created. The branching process theory in Section \ref{s:BP} suggests that when a collision edge occurs,
the generation of both vertices that are part of the collision edge satisfies a central limit theorem. Further, the residual time to birth of the active half-edge to which we have paired the newly found half-edge converges in distribution to the residual life-time distribution. Thus, we only need to argue that the stochastic process that describes the times of finding the collision edges and centered by  $\tn$ as in \eqref{eq:resc-variables}
converges to a PPP with intensity measure $t\mapsto \frac{2\Ev[D^\star] \fR(0)}{\expec[D]}\e^{2\lambda t}$.
For this, we note that the rate at which new half-edges are found at time $t+\tn$ is roughly equal to
$2\fR(0) |\cA_n(t+\tn)||\widetilde \cA_n(t+\tn)| /\CL_n$, where the factor $\fR(0)$ is due to the fact that
half-edges with remaining life-time equal to 0 are the ones to die, and the factor 2 due to the fact that $\cF_n$ as well as 
$\widetilde\cF_n$ can give rise of the birth of the half-edge.  

Here we also note that $|\cA_n(t+\tn)|$ and 
$|\widetilde \cA_n(t+\tn)|$ are of order $\sqrt{n}$, and thus the total number of half-edges is 
equal to $\CL_n(1+\op(1))$.) When a half-edge dies, it has a random number of children with distribution close to $D_n^{\star}$, and each of the corresponding half-edges can create a collision edge, hence we add an extra $\Ev[D_n^\star]$ factor. Further, we can approximate $\CL_n\approx n\expec[D_n]$, $|\cA_n(t)|\approx \e^{\lambda_n t} W_{s_n}$ and 
$|\widetilde \cA_n(t)|\approx \e^{\lambda_n t} \widetilde W_{s_n}$, so that,
using \eqref{tn-def},
	\be 
	\frac{
	\Ev[D_n^\star]\fR(0)|\cA_n(t+\tn)||\widetilde \cA_n(t+\tn)|}{\CL_n}
	\approx \frac{
	\Ev[D_n^\star]\fR(0)}{\expec[D]n} \e^{2\lambda_n (t+\tn)} W_{s_n}\widetilde W_{s_n}\\
	=\frac{\expec[D^{\star}_n] f_R(0)
	}{\expec[D_n]} \e^{2\lambda_n t}.
	\ee
This explains Theorem \ref{thm-main-PPP}.

\subsection{The Barbour-Reinert connection process - differences} 
\label{ss:BRconnect}
The main difference between the Barbour-Reinert proof of Proposition \ref{prop:connection} and the previous section is that in the proof in \cite{BarRei13}, the forward and the backward cluster are run after each other, not simultaneously:

We couple the infection process together with the exploration on $\CMD$ to the forward BP with small errors up to time $t^\star_n:=\tau_{\sqrt{n}}$ in the forward process ($\tau_{\sqrt{n}}$ denotes the time when the $\sqrt{n}$th vertex enters the infection), which we \emph{freeze} after this time. We then couple the backward process \emph{conditionally on the frozen cluster of forward process} up to time $\frac{1}{2\la_n} \log n + K$ time for some large $K>0$. Then, by \eqref{eq:asalives} we see that for any $u \in \R$, at time $t_n(u):=\frac{1}{2\la_n} \log n +u$ 
the size of the coming generation in the forward process is $|\CA_{\tau_{\sqrt{n}}}|=c_n^A \sqrt{n}(1+o(1))$ for a specific constant $c_n^A$
and the size of the backward cluster is
 $|\widetilde \CA_{t_n(u)}| =c_n^A \sqrt{n} \e^{\la u} \widetilde W_{t_n(u)}(1+o(1)).$ 
 From here the formation of collision edges leads to a similar two dimensional Poisson process to the one described as the first and last coordinate in \eqref{intensity-PPP}, i.e. here the intensity measure, conditioned on $\widetilde W_{s_n}$ is given by 
 	\[
	\frac{
	\Ev[D^\star] \fR(0)}{\expec[D]}\e^{\lambda x} 
	\widetilde W_{s_n} \mathrm dx \otimes F_R(\mathrm dy).
	\]
From here onwards, the two proofs are essentially the same: the factor $W$ from the forward process appears it the formula $\tau_{\sqrt{n}}\approx \frac1{2\la_n}\log n - \frac{1}{\la} \log W_{s_n}$. The minimisation problem \eqref{rewrite-Gumbel} is then solved by  calculating the probability that there are no PPP points in the infinite triangle $x+y\le t$, yielding the statement of Proposition \ref{prop:connection}.

\subsection{Proof of Theorem \ref{thm-epidemic-curve}}
\label{ss:proof_main_theorem}
In Sections \ref{ss:BHHconnect} and \ref{ss:BRconnect} we gave two possible ways to determine the length of the shorts infection path between two uniformly chosen vertices. 
Now we use Proposition \ref{prop:connection}
to explain how to get the epidemic curve in Theorem \ref{thm-epidemic-curve} and complete its proof.

The proof of Theorem \ref{thm-epidemic-curve} will be based on the following key proposition that we prove below. Let $s_n\rightarrow \infty$ as in Proposition \ref{xxx},  and denote $W_{s_n}=\e^{-s_n\lambda_n}|\CA_{s_n}|$, where, as before, $|\CA_{t}|$ is the size of the coming generation of infected individuals at time $t$.

\begin{Proposition}[The epidemic curve with an offset]
\label{prop-epi-offset}
Under Condition \eqref{cond:CMFinVar}, consider the epidemic spread with i.i.d.\ continuous infection times on the configuration model $\CMD$. For every $t>0$,
	\eqn{
		\label{Pn-offset-conv}
	\Pn\Big(t+\frac{\log n}{\la_n}-\frac{\log{W_{s_n}}}{\lambda_n}, \alpha_n\log{n}+x\sqrt{\beta\log{n}}\Big)
	\convp P(t)\Phi(x).
	}\end{Proposition}
\bigskip
\noindent
{\it Proof of Theorem \ref{thm-epidemic-curve} subject to Proposition \ref{prop-epi-offset}.}
Fix $x\in {\mathbb R}$. Since $t\mapsto \Pn(t,v)$ is non-decreasing, and since the limit $t\mapsto P(t)$ in
\eqref{Pn-offset-conv} is non-decreasing, continuous and bounded, Proposition \ref{prop-epi-offset}
implies that the covergence in \eqref{Pn-offset-conv}  holds {\em uniform} in $t$, i.e., we have
	\eqn{
	\sup_{s\in {\mathbb R}} \Big|\Pn\Big(s+\frac{\log n}{\la_n}-\frac{\log W_{s_n}}{\la_n}, \alpha_n\log{n}+x\sqrt{\beta\log{n}}\Big)-P(s)\Phi(x)\Big|
	\convp 0.
	}
Applying this to $s=t+\frac{\log W_{s_n}}{\la_n}$, we thus obtain that
	\eqn{
	\Pn\Big(t+\frac{\log n}{\la_n}, \alpha_n\log{n}+x\sqrt{\beta\log{n}}\Big)=P(t+\frac{\log W_{s_n}}{\la_n})\Phi(x)+\op(1).
	}
Since $\frac{\log W_{s_n}}{\la_n}\convd \frac{\log{W}}{\lambda}=-S$ and $t\mapsto P(t)$ in continuous, this completes the proof of Theorem \ref{thm-epidemic-curve}.
\qed
\bigskip

\noindent
\begin{proof}[Proof of Proposition \ref{prop-epi-offset}.] We next complete the proof of Proposition \ref{prop-epi-offset} using Proposition \ref{prop:connection}. 
We perform a second moment method on
$\Pn\Big(t+\frac{\log n}{\la_n}-\frac{\log W_{s_n}}{\la_n}, \alpha_n\log{n}+x\sqrt{\beta\log{n}}\Big)$, conditionally on $\cG_n(s_n)$. To simplify notation, we will take $x=\infty$, $S_n=-\frac{\log W_{s_n}}{\la_n}, \widetilde S_n=-\frac{\log \widetilde W_{s_n}}{\la_n}$ and show that
	\eqn{
	\label{sec-cond-mom}
	\expec\Big[\Pn\Big(t+\log n/\la_n+S_n\Big)\mid \cG_n(s_n)\Big]
	\convp P(t),
	\qquad
	\expec\Big[\Pn\Big(t+\log n/\la_n+S_n\Big)^2\mid \cG_n(s_n)\Big]
	\convp P(t)^2.
	}
Equation \eqref{sec-cond-mom} implies that, conditionally on $\cG_n(s_n)$,
$\Pn\Big(t+\log n/\la_n+S_n\Big)\convp P(t)$, as required.
We start by identifying the first conditional moment. For this, we note that
	\eqan{
	\expec\Big[\Pn\Big(t+\log n/\la_n+S_n\Big)\mid \cG_n(s_n)\Big]
	&=\frac{1}{n}\sum_{w\in[n]}
	\prob\Big(\Wn(V_n,w) \le  t +\log n/\la_n+S_n\mid \cG_n(s_n)\Big)\nn\\
	&=\prob\Big(\Wn(V_n,\widetilde V_n^{\sss(1)})-\log n/\la_n-S_n\leq t\mid \cG_n(s_n)\Big),
	}
where $\widetilde V_n^{\sss(1)}$ is a uniform vertex independent of $V_n$ and $\Wn(v,w)$ is the time that the infection starting from $v$ reaches $w$. Thus, in the infinite-contagious period case, $\Wn(v,w)$ is nothing but the first-passage time from $v$ to $w$. For $s>0$, let $\widetilde \cG_n^{\sss(1)}(s)$
denote the $\sigma$-algebra of all vertices that would infect $\widetilde V_n^{\sss(1)}$ within time $s$ if the infection started from them at time $0$, as well as all edge weights of the edges that are incident to such vertices. Thus, by the argument about the backward process in Section \ref{ss:forwardbackward}, these vertices are the same as the vertices that would be infected before time $s$ from an infection started from $\widetilde V_n^{\sss(1)}$ in the backward process.

Write $\widetilde W_{s_n}=\e^{-\lambda_n s_n}|\widetilde\CA_{s_n}|$, where $\widetilde \CA_{t}$ denotes half-edges that are in the coming generation of the backward infection process of $\widetilde V_n^{\sss(1)}$ at time $t$.  We now further condition on $\widetilde \cG_n^{\sss(1)}(s)$, and obtain
	\eqan{
	\expec\Big[\Pn\Big(t+\log n/\la_n+S_n\Big)\mid \cG_n(s_n)\Big]
	=\expec\Big[\prob\Big(\Wn(V_n,\widetilde V_n^{\sss(1)})-\log n/\la_n-S_n\leq t\mid \cG_n(s_n), \widetilde \cG_n^{\sss(1)}(s_n)\Big)\mid \cG_n(s_n)\Big].
	}
By Proposition \ref{prop:connection} 
there exists a constant $c>0$ such that
	\eqn{
	\label{conv-double-cond-prob}
	\prob\Big(\Wn(V_n,\widetilde V_n^{\sss(1)})-\log n/\la_n
	-S_n-\widetilde S_n\leq t\mid \cG_n(s_n), \widetilde \cG_n^{\sss(1)}(s_n)\Big)
	\convp \prob(-\Lambda/\lambda + c\leq t).
	}

Again, since $t\mapsto \prob(-\Lambda/\lambda + c\leq t)$ is increasing and continuous, the above convergence even holds uniformly in $t$, i.e.,
	\eqn{
	\sup_{t\in {\mathbb R}} 
	\Big|\prob\Big(\Wn(V_n,\widetilde V_n^{\sss(1)})-\log n/\la_n
	-S_n-\widetilde S_n\leq t
	\mid \cG_n(s_n), \widetilde \cG_n^{\sss(1)}(s_n) \Big)-\prob(-\Lambda/\lambda + c\leq t)\Big|\convp 0.
	}
As a result,  
	\eqan{
	\expec\Big[\Pn\Big(t+\log n/\la_n+S_n\Big)\mid \cG_n(s_n), \widetilde \cG_n^{\sss(1)}(s_n)\Big]
	&=\prob\Big(\Wn(V_n,\widetilde V_n^{\sss(1)})-\log n/\la_n
	-S_n\leq t
	\mid \cG_n(s_n), \widetilde \cG_n^{\sss(1)}(s_n)\Big)\\
	&=\prob(-\Lambda/\lambda + c\leq t-\widetilde S_n\mid \widetilde \cG_n^{\sss(1)}(s_n)) + \op(1),\nn
	}
and since $\widetilde W_{s_n}\convp \widetilde W$ and $t\mapsto \prob(-\Lambda/\lambda + c\leq t)$ is continuous and bounded, we obtain that
	\eqn{
	\expec\Big[\Pn\Big(t+\log n/\la_n+S_n\Big)\mid \cG_n(s_n)\Big]
	\convp \prob(-\Lambda/\lambda + c\leq t-\widetilde S)
	=P(t).
	}
By bounded convergence, this also implies that
	\eqn{
	\expec\Big[\Pn\Big(t+\log n/\la_n+S_n\Big)\mid \cG_n(s_n)\Big]
	\convp P(t),
	}
which completes the proof of the convergence of the first moment.

We use similar ideas to identify the second conditional moment, for which we start by writing
	\eqan{
	&\expec\Big[\Pn\Big(t+\log n/\la_n+S_n\Big)^2\mid \cG_n(s_n)\Big]\\
	&=\frac{1}{n}\sum_{i,j\in[n]}
	\prob\Big(\Wn(V_n,i)+\log n/\la_n+S_n\leq t,
	\Wn(V_n,j)+\log n/\la_n+S_n\leq t\mid \cG_n(s_n)\Big)\nn\\
	&=\prob\Big(\Wn(V_n,\widetilde V_n^{\sss(1)})+\log n/\la_n+S_n\leq t,
	\Wn(V_n,\widetilde V_n^{\sss(2)})+\log n/\la_n+S_n\leq t\mid \cG_n(s_n)\Big),\nn
	}
where $V_n, \widetilde V_n^{\sss(1)}, \widetilde V_n^{\sss(2)}$ are three i.i.d.\  uniform vertices in $[n]$. 
For $s>0$ and $j\in \{1,2\}$, let $\widetilde \cG_n^{\sss(j)}(s)$
denote the $\sigma$-algebra of all vertices that would infect $\widetilde V_n^{\sss(j)}$ within time $s$ if the infection started from them at time 0, as well as all edge weights of the edges that are incident to such vertices. Thus, these vertices are the same as the vertices that would be infected before time $s$ in the backward infection process started from $\widetilde V_n^{\sss(j)}$. 

Write $\widetilde W_{s_n}^{\sss(j)}= \e^{-\lambda_n s_n}|\widetilde\CA^{\sss(j)}_{s_n}|,\  \widetilde S_n^{\sss{(i)}}=-\log \widetilde W^{\sss{(i)}}/\la_n$. We now further condition on $\widetilde \cG_n^{\sss(1)}(s_n)$ and $\widetilde \cG_n^{\sss(2)}(s_n)$, and obtain 
	\eqan{
	&\expec\Big[\Pn\Big(t+\log n/\la_n-S_n\Big)\mid \cG_n(s_n)\Big]\\
	&=\expec\Big[\prob\Big(\Wn(V_n,\widetilde V_n^{\sss(1)})-\log n/\la_n-S_n\leq t,
	\Wn(V_n,\widetilde V_n^{\sss(2)})-\log n/\la_n-S_n\leq t \mid \cG_n(s_n), \widetilde \cG_n^{\sss(1)}(s_n),\widetilde \cG_n^{\sss(2)}(s_n)\Big)\mid \cG_n(s_n)\Big].\nn
	}
By \eqref{eq:shortest-path-prob-two} in Proposition \ref{prop:connection}, there exists a constant $c>0$ such that
	\eqan{
	\label{conv-triple-cond-prob}
	&\prob\Big(\Wn(V_n,\widetilde V_n^{\sss(i)})-\log n/\la_n-S_n-\widetilde S_n^{\sss{(i)}}\leq t,~ i=1,2\mid \cG_n(s_n), \widetilde \cG_n^{\sss(1)}(s_n),\widetilde \cG_n^{\sss(2)}(s_n)\Big)\nn\\
	&\qquad\convp \prob(-\Lambda/\lambda + c\leq t, -\Lambda'/\lambda + c\leq t)^2=\prob(-\Lambda/\lambda + c\leq t)^2,
	}
since $\Lambda, \Lambda'$ are two independent Gumbel variables.
Now the argument for the first moment can be repeated to yield
	\eqn{
	\expec\Big[\Pn\Big(t+\log n/\la_n+S_n\Big)^2\mid \cG_n(s_n)\Big]
	\convp P(t)^2,
	}
which completes the proof of the convergence of the second moment for $x=\infty$.

The extension to $x<\infty$ follows in an identical fashion, now using that by \cite[Theorem 2.2]{BhaHofHoo12},
	\eqan{
	\label{conv-double-cond-prob-Hn}
	&\prob\Big(\Wn(V_n,\widetilde V_n^{\sss(1)})-\log n/\la_n
	-S_n-\widetilde S_n\leq t, H_n(V_n,\widetilde V_n^{\sss(1)})\leq \alpha_n\log{n}+x\sqrt{\beta\log{n}}\mid \cG_n(s_n), \widetilde \cG_n^{\sss(1)}(s_n)\Big)\nn\\
	&\qquad\qquad\convp \prob(-\Lambda/\lambda + c\leq t)\Phi(x),
	}
as well as a three vertex extension involving $V_n,\widetilde V_n^{\sss(1)}$ and $\widetilde V_n^{\sss(2)}$.
We omit further details.
\end{proof}

\section*{Acknowledgements}
The work of RvdH and JK is supported in part by The Netherlands Organisation for
Scientific Research (NWO). SB has been partially supported by NSF-DMS grants 1105581 and 1310002.

\bibliographystyle{abbrv}
\bibliography{epidemic-curve}

\end{document}